\documentclass[final,1p,times]{elsarticle}
\usepackage{amsmath, amsthm}
\usepackage{amssymb}
\usepackage{hyperref}
\usepackage{esint}
\usepackage{color}
\usepackage{mathrsfs}
\usepackage{bbold}
\usepackage{enumerate}

\newcommand{\Om}{\Omega}
\newcommand{\lb}{\lambda}
\newcommand{\ve}{\varepsilon}
\newcommand{\sg}{\sigma}

\newcommand{\fun}[6]{\ensuremath{\Phi^{#1,#2}_{#3,#4}(#6,#5)}}

\makeatletter
\providecommand*{\dif}%
   {\@ifnextchar^{\DIfF}{\DIfF^{}}}
\def\DIfF^#1{%
   \mathop{\mathrm{\mathstrut d}}%
      \nolimits^{#1}\gobblespace
}
\def\gobblespace{%
   \futurelet\diffarg\opspace}
\def\opspace{%
   \let\DiffSpace\!%
   \ifx\diffarg(%
      \let\DiffSpace\relax
   \else
      \ifx\diffarg\[%
         \let\DiffSpace\relax
      \else
         \ifx\diffarg\{%
            \let\DiffSpace\relax
         \fi\fi\fi\DiffSpace}

\newcommand{\medxr}{\mu                                               B(x,r)}
\newcommand{\triplet}{\ensuremath{(X,d,\mu)} }
\newcommand{\norm}[1]{\ensuremath{\left\|#1 \right\|}}

\theoremstyle{plain}
\newtheorem{theorem}{Theorem}[section]

\newtheorem{lemma}[theorem]{Lemma}
\newtheorem{proposition}[theorem]{Proposition}
\newtheorem{corollary}[theorem]{Corollary}

\theoremstyle{definition}
\newtheorem{definition}[theorem]{Definition}

\theoremstyle{remark}
\newtheorem{remark}[theorem]{Remark}

\newcommand{\ggMs}[5]{\ensuremath{L^{#1), #2)}_{#3,#4}  \left(#5, \mu\right)}}
\newcommand{\Ms}[2]{\ensuremath{L^{#1, #2}  \left(X,\mu\right)}}
\newcommand{\bmo}{\ensuremath{\mathrm{BMO}(X,\mu)}}
\newcommand{\vmo}[1]{\ensuremath{\mathrm{VMO}(#1)}}
\newcommand{\com}[2]{\ensuremath{[#1,#2]}}
 \usepackage{amsthm}

\journal{TO CHOOSE}

\begin{document}
\begin{frontmatter}



\title{Boundedness of Commutators of Singular and Potential Operators in  Generalized  Grand Morrey Spaces and some applications}

\author[add1,add2]{Vakhtang Kokilashvili}
\ead{kokil@rmi.ge}

\address[add1]{Department of Mathematical Analysis, A. Razmadze Mathematical Institute, I. Javakhishvili Tbilisi State University,
2. University Str., 0186 Tbilisi, Georgia}
\address[add2]{International Black Sea University, 3 Agmashenebeli Ave., Tbilisi 0131, Georgia }

\author[add1,add3,add4]{Alexander Meskhi}
\address[add3]{Department of Mathematics,  Faculty of Informatics and Control
Systems,  Georgian Technical University, 77, Kostava St., Tbilisi, Georgia}
\address[add4]{Abdus Salam School of Mathematical Sciences, GC
University, 68-B New Muslim Town, Lahore, Pakistan}
\ead{meskhi@rmi.ge}

\author[add5,add6]{Humberto Rafeiro}
\address[add5]{Instituto Superior T\'ecnico, Dep. de Matem\'atica, Centro CEAF, Av. Rovisco Pais, 1049--001 Lisboa, Portugal}
\address[add6]{Pontificia Universidad Javeriana, Dep. de Matem\'aticas, Cra 7a No 43-82 Ed. Carlos Ortiz 604, Bogot\'a, Colombia }
\ead{hrafeiro@math.ist.utl.pt}
\ead{silva-h@javeriana.edu.co}

\begin{abstract}
In the setting of homogeneous spaces $(X, d,\mu)$, it is shown that the commutator of  Calder\'on-Zygmund type operators as well as commutator of potential operator with BMO function are bounded in generalized Grand Morrey space. Interior estimates for solutions of elliptic equations are also given in the framework of generalized grand Morrey spaces.
\end{abstract}


\end{frontmatter}



\section{ Introduction}

In 1992 T. Iwaniec and C. Sbordone \cite{iwa_sbor1992},  in their studies related with the integrability properties of the Jacobian in a bounded open set $\Omega,$ introduced a new type of function spaces  $L^{p)}(\Omega),$   called \textit{grand Lebesgue spaces}. A generalized version of them, $L^{p),\theta}(\Omega)$ appeared in  L. Greco, T. Iwaniec and C. Sbordone \cite{greco}.
 Harmonic analysis related to these spaces and  their associate spaces (called \textit{small Lebesgue spaces}), was intensively studied during last years due to various applications, we mention e.g. \cite{capone_fio, fratta_fio, fio200, fio_gupt_jain, fio_kara, fio_rako, kok_2010}.

Recently in \cite{samko_umar} there was  introduced a version of weighted grand Lebesgue spaces adjusted for sets $\Om\subseteq \mathbb{R}^n$  of infinite measure, where  the integrability of
$|f(x)|^{p-\ve}$ at infinity was controlled by means of a weight, and there grand grand Lebesgue spaces were also considered, together with the study of classical operators of harmonic analysis in such spaces.  Another idea of introducing ``bilateral" grand Lebesgue spaces on sets of infinite measure was suggested in \cite{357ad}, where the structure of such spaces was investigated, not operators; the spaces in \cite{357ad} are two parametrical with respect to the exponent $p$, with the norm involving $\sup_{p_1<p<p_2}.$ \\

Morrey spaces $L^{p,\lambda}$ were introduced in  1938 by C. Morrey \cite{405a}  in relation to regularity problems of solutions to partial differential equations, and provided a  useful tool in the regularity theory of PDE's (for  Morrey spaces we refer to the books \cite{187a, kuf}, see also  \cite{rafsamsam} where an overview of various generalizations may be found).

 Recently, in the spirit of grand Lebesgue spaces, A. Meskhi \cite{meskhi2009, meskhi} introduced \textit{grand Morrey spaces}  (in \cite{meskhi2009} it was already defined on quasi-metric measure spaces with doubling measure) and obtained results on the boundedness   of the maximal operator, Calder\'on-Zygmund singular operators and Riesz potentials. The boundedness of commutators of singular and potential operators in grand Morrey spaces was already treated by X. Ye \cite{ye_xf}.  Note that  the ``\textit{grandification} procedure" was applied only to the parameter $p.$ \\

This paper is a continuation of the work began in \cite{rafeiro2012} and \cite{kokmesraf}, where in the former the introduction of generalized grand Morrey spaces (in that paper they where called grand grand Morrey spaces) and the study of maximal and Calder\'on-Zygmund operators was done in the framework of the Euclidean spaces whereas in the latter paper the study of the  boundedness of potential operators was done in the framework of generalized grand Morrey spaces in homogeneous and even in the nonhomogeneous case.\\

\noindent \textbf{Notation:}

\noindent    $d_X$    denotes    the    diameter    of    the    $X$   set;\\
\noindent     $A     \sim     B$     for     positive     $A$     and     $B$
means  that  there  exists  $c>0$  such  that  $c^{-1}A  \leqslant B \leqslant c A$;\\
\noindent$B(x,r)=\{y\in             X:             d(x,y)<r            \}$;\\
\noindent $A\lesssim B$ stands for $A \leqslant C B$;\\
\noindent by $c$ and $C$ we denote various absolute positive constants,
which may have different values even in the same line;\\
$\hookrightarrow$ means continuous imbedding;\\
\noindent $\fint_B f \dif \mu$ denotes the integral average of $f$, i.e. $\fint_B f \dif \mu:= \frac{1}{\mu B} \int_B f \dif \mu$;\\
\noindent $p^\prime$ stands for the conjugate exponent $1/p+1/p^\prime=1$.

\section{Preliminaries}

\subsection{Spaces of homogeneous type}\label{preliminaries}
Let $X:=\triplet$ be a topological space with a complete measure $\mu$ such that the space of compactly supported continuous functions is dense in $L^1(X,\mu)$ and $d$ is a quasimetric, i.e.  it is  a non-negative real-valued function $d$ on $X\times X$ which  satisfies the conditions:
\begin{enumerate}
\item[(i)] $d(x,y)=0$ if and only if $x=y$;
\item[(ii)] there exists a constant $C_t>0$ such that $d(x,y)\leqslant  C_t  [d(x,z)+d(z,y)]$, for all $x,y,z \in X$, and
\item[(iii)] there exists a constant $C_s>0$ such that $d(x,y)\leqslant  C_s \cdot d(y,x)$, for all $x,y \in X$.
\end{enumerate}
Let  $\mu$  be  a  positive measure on the $\sigma$-algebra of subsets of $X$ which  contains  the  $d$-balls $B(x,r).$
Everywhere in the sequel we assume that all the balls have a finite measure, that is, $\mu B(x,r)<\infty$ for all
$x\in X$ and $r>0$ and that for every neighborhood $V$ of $x\in X$, there exists $r>0$ such that $B(x,r)\subset V$.

We  say  that the measure $\mu$ is \textit{lower  $\alpha$-Ahlfors regular},
if
\begin{equation}\label{lowerahlforscondition}
 \medxr\geqslant                                                         cr^\alpha
\end{equation}
and \textit{upper  $\beta$-Ahlfors regular} (or, it satisfies the  \textit{growth condition of degree $\beta$}), if
\begin{equation}\label{upperahlforscondition}
 \mu B(x,r)\le cr^\beta,
\end{equation}
where  $\alpha,\beta, c >0$ does not depend on $x$ and $r$.
 When $\alpha=\beta$,
the measure  $\mu$ is simply called \textit{$\alpha$-Ahlfors regular}.\\

The  condition
\begin{equation}\label{doublingcondition}
 \mu    B(x,2r)\leqslant   C_d \cdot \medxr,   \quad   C_d   >   1
\end{equation}
on the measure $\mu$ with $C_d $ not depending on $x\in X$ and
$0<r<d_X$, is known as the \textit{doubling condition}. 

Iterating     it,      we      obtain
\begin{equation}\label{doublingcondition2}
 \frac{\mu B(x,R)}{\mu B(y,r)} \leqslant C_d \left( \frac{R}{r}\right)^{\log_2 C_d}, \quad  0< r\leqslant R
\end{equation}
for  all  $d$-balls  $B(x,R)$ and $B(y,r)$ with $B(y,r)\subset B(x,R)$.

The triplet $(X,d,\mu)$, with  $\mu$ satisfying the doubling condition, is
called a  \textit{space  of  homogeneous  type}, abbreviated from now on simply as SHT. For some important examples of an SHT we refer e.g. to \cite{coifmanweiss}.

From  \eqref{doublingcondition2}   it   follows   that every homogeneous  type space \triplet with   finite  measure is lower  $(\log_2 C_d)$-Ahlfors regular. \\


Throughout the paper we will also assume the following condition
\begin{equation}\label{eq:annuluspositivemeasure}
\mu (B(x,R)\backslash B(x,r))>0
\end{equation}
for all $x \in X$ and $r, R$ with $0 < r < R < d_X$.
The validity of the reverse doubling condition, following from the doubling condition under certain restrictions, is well known (cf., for example, \cite[p. 269]{18}).  For example, when \eqref{eq:annuluspositivemeasure} is valid and \triplet is an SHT, then the measure $\mu$ also satisfies the reverse doubling condition
\begin{equation}\label{eq:rd}
\frac{\mu B(x,r)}{\mu B(x,R)}\leqslant C \left(\frac{r}{R}\right)^\gamma
\end{equation}
for appropriate positive constants $C$ and $\gamma$. For other conditions dealing with the validity of the reverse doubling condition whenever the measure is doubling, see, e.g. \cite{rafsam}.


\subsection{Generalized Lebesgue spaces} For $1<p<\infty$, $\theta >0$ and $0<\ve<p-1$ the \emph{grand Lebesgue space} is the set of measurable functions for which
\begin{equation}\label{added}
 \|f\|_{L^{p),\theta}(X,\mu)}:=\sup_{0<\varepsilon<p-1} \varepsilon^\frac{\theta}{p-\varepsilon} \|f\|_{L^{p-\varepsilon}(X,\mu)}<\infty,
 \end{equation}
where $\|f\|^p_{L^p(X,\mu)}:=\int_X |f(y)|^p \dif \mu(y)$.
In the case $\theta=1,$ we denote $L^{p),\theta}(X,\mu):=L^{p)}(X,\mu).$

When $\mu X<\infty$, then for all $0<\ve< p-1<\theta_{1}<\theta_{2}$ we have
\[
L_{w}^{p}(X,\ \mu)\hookrightarrow L_{w}^{p),\theta_{1}}(X,\ \mu)\hookrightarrow L_{w}^{p),\theta_{2}}(X,\ \mu)\hookrightarrow L_{w}^{p-\in}(X,\ \mu),
\]
where $w$ is a Muckenhoupt weight.

For more properties of grand Lebesgue spaces, see \cite{kok_2010}.

\subsection{Morrey spaces}
For $1\leqslant p< \infty $ and $0\leqslant \lambda <1$, the usual Morrey space $L^{p,\lambda}(X,\mu)$ is introduced as the set of all measurable functions such that
\begin{equation}\label{eq:morrey_norm}
\|f\|_{L^{p,\lambda}(X,\mu)}:=\sup_{\stackrel{x \in X}{0<r< d_X }}\left(\frac{1}{\medxr ^{\lambda}}   \int_{B(x,r)} |f(y)|^{p} \dif \mu(y)\right)^\frac{1}{p}<\infty.
\end{equation}

\subsection{BMO space} The space of \textit{bounded mean oscillation}, denoted by \bmo,  is the set of all real-valued locally integrable functions such that
\begin{equation}\label{eq:bmo}
\norm{f}_{\bmo}=\sup_{\substack{x\in X \\ 0<r<d_X}}\frac{1}{\mu B(x,r)} \int_{B(x,r)} |f(y)-f_{B(x,r)}|\dif \mu(y)<\infty,
\end{equation}
where $f_{B(x,r)}$ is the integral average over the ball $B(x,r)$. \bmo \; is a Banach space with respect to the norm $\norm{\cdot}_\bmo$ when we regard the space \bmo\; as the class of equivalent functions modulo additive constants.

\begin{remark}\label{rem:2.1}
 In this remark, we give equivalent norms for $\bmo$-functions, namely

\begin{enumerate}[(i)]
\item we can define an  equivalente norm in $\bmo$  as
\begin{equation}\label{eq:bmo1}
\norm{f}_{\bmo}\sim \sup_{\substack{x\in X \\0<r<d_X}}\inf_{c\in \mathbb R}\frac{1}{\mu B(x,r)}\int_{B(x,r)}|f(y)-c|\dif \mu(y),
\end{equation}
\item the John-Nirenberg inequality give us another  equivalent norm for $\bmo$-functions given by
\begin{equation}\label{eq:bmoJN}
\norm{f}_{\bmo}\sim \sup_{\substack{x\in X \\0<r<d_X}} \left(\frac{1}{\mu B(x,r)}\int_{B(x,r)} |f(y)-f_{B(x,r)}|^p \dif\mu(y) \right)^\frac{1}{p}
\end{equation}
valid for $1<p<\infty$, where $f_B$ stands for the integral average.
\end{enumerate}
\end{remark}

\subsection{Maximal operators}
We denote  by $Mf$ the \textit{Hardy-Littlewood maximal operator}, given by
\begin{equation}\label{eq:ms}
Mf(x)=\sup_{ 0< r<d}\fint_{B(x,r)}|f(y)|\dif\mu(y),
\end{equation}
for $x \in X$.

From \cite{meskhi} we have the following boundedness result for Morrey spaces.
\begin{lemma}  \label{lemma:2.1}
Let $ 1<p<\infty$ and $0\leqslant \lambda<1$.  Then
\[
\norm{Mf}_{L^{p, \lambda}(X,\mu)}\leqslant\left(Cb^{\lambda/p}(p')^{\frac{1}{p}}+1\right)\norm{f}_{L^{p,\lambda}(X,\mu)}
\]
holds, where the constant $b\geqslant 1$ arises in the doubling condition for $\mu$ and $C$ is the constant independent of $p$.
\end{lemma}

By $M_s f$ we define the following \textit{maximal operator}
\[
M_s f(x):=\left(M|f|^s \right)^\frac{1}{s}
\]
for $1\leqslant s<\infty$.  Using Lemma \ref{lemma:2.1}, it is easy to obtain that the following boundedness result.

\begin{lemma}\label{lem:2.2}
Let $ 1<s<p<\infty$ and $0\leqslant\lambda<1$.  Then
\[
\norm{M_{s}f}_{L^{p,\lambda}(X,\mu)} \leqslant \left(C b^{\lambda s/p}\big( (p/s)'\big) ^{\frac{\mathrm{s}}{p}}+1 \right) \norm{f}_{L^{p,\lambda}(X,\mu)}.
\]
holds, where the constant $b\geqslant 1$ arises in the doubling condition for $\mu$ and $C$ is the constant independent of $p$.
\end{lemma}
When $\lambda=0$, we have $L^{p}(X, \mu)$ boundedness of $M_{s}f$. \\

We will also need another maximal operator, namely the so-called \textit{sharp maximal operator},

\begin{definition}[Sharp maximal function] For all locally integrable function $f$ and $x\in X$, we denote  the  sharp maximal function $f^\sharp(x)$  by
\[
f^\sharp(x)=\sup_{0<r<d_x}\frac{1}{\mu(B(x,r))}\int_{B(x,r)}|f(y)-f_{B(x,r)}|\dif
\mu(y).
\]
\end{definition}
It is immediate from the definition of the sharp maximal function that it is a.e. pointwise dominated by the maximal function, $f^\sharp (x)\leqslant 2 Mf(x)$, but we also have some relation in the other direction,  namely we have the so-called \textit{Fefferman-Stein inequality} proved in the case of Lebesgue spaces in the Euclidean setting in \cite{FeSt2}. Our version is taken from \cite{liu}, namely

\begin{lemma}\label{lemma:feffermanstein}
Let $1<p<\infty$ and let $0\leqslant \lambda <1$. Then
\[
\norm{Mf}_{\Ms{p}{\lambda}}\leqslant C(b^{\lambda/p}+1)\norm{f^\sharp}_{\Ms{p}{\lambda}}.
\]
\end{lemma}

\subsection{Calder\'on-Zygmund singular operators}
We follow \cite{meskhi} in this section, in particular, making use of the following definition of the  Calder\'on-Zygmund singular operators.
Namely, the  Calder\'on-Zygmund  operator is defined as the integral operator
\[
Tf(x)=\mathrm{p.v.} \int_X K(x,y)f(y)\;\mathrm{d}\mu(y)
\]
with the kernel  $K:X \times X \backslash \{(x,x): x \in \Omega\} \to \mathbb R$ being a measurable function satisfying the conditions:
\begin{enumerate}[(i)]
\item $|K(x,y)|\leqslant \frac{C}{\mu B(x,d(x,y))}, \quad x,y \in X, \quad x\neq y$;
\item $\displaystyle |K(x_1,y)-K(x_2,y)|+|K(y,x_1)-K(y, x_2)| \leqslant C w\left(\frac{d(x_2,x_1)}{d(x_2,y)}\right) \frac{1}{\mu B(x_2,d(x_2,y))}$
\end{enumerate}
for all $x_1$, $x_2$ and $y$ with $d(x_2,y)\geqslant Cd(x_1,x_2)$, where $w$ is a positive non-decreasing function on $(0,\infty)$ which satisfies the $\Delta_2$ condition $w(2t)\leqslant c w(t)$ ($t>0$) and the Dini condition $\int_0^1 w(t)/t \;\mathrm{d}t<\infty$.
We also assume that  $Tf$ exists almost everywhere on $X$  in the principal value sense for all $f \in L^{2}(X)$ and that $T$ is bounded in $L^{2}(X).$ \\
The boundedness of such Calder\'on-Zygmund operators in Morrey spaces is valid, as can be seen in the following Proposition, proved in \cite{meskhi}.
\begin{proposition}\label{prop:boundednes_ron}Let
 $1<p<\infty$ and $0\leqslant \lambda <1$. Then
\[
\|Tf\|_{L^{p,\lambda}(X,\mu)} \leqslant C_{p,\lambda} \|f\|_{L^{p,\lambda}(X,\mu)}
\]
where
\begin{equation}\label{equ:constant_Morrey_calderon}
C_{p,\lambda}\leqslant  c\left\{
  \begin{array}{ll}
     \frac{p}{p-1}+\frac{p}{2-p} +\frac{p-\lambda+1}{1-\lambda} & \mbox{if } 1<p<2, \\
    p+\frac{p}{p-2} +\frac{p-\lambda+1}{1-\lambda} & \mbox{if } p>2,\\
  \end{array}
\right.
\end{equation}
with $c$ not depending on $p$ and $\lambda$.
\end{proposition}

\subsection{Commutators}

Let $U$ be an operator and $b$ a locally integrable function. We define the \textit{commutator} $\com{b}{U}f$ as
\[
\com{b}{U}f=bU(f)-U(bf).
\]
Commutators are very useful when studying problems related with regularity of solutions of
elliptic partial differential equations of second order, e.g., \cite{[2]}, \cite{[2']}\\

\section{Generalized grand Morrey spaces and the reduction lemma}
In this section we will assume that the measure $\mu$ is upper $\gamma$-Ahlfors regular.
All the stated results in this section were proved in \cite{kokmesraf}. \\

We introduce the following functional
 \begin{equation}\label{added}
 \Phi^{p,\lb}_{\varphi,A}(f,s):=\sup_{0<\varepsilon<s} \varphi(\varepsilon)^\frac{1}{p-\varepsilon} \|f\|_{L^{p-\varepsilon,\lambda -A(\varepsilon)}(X,\mu)},
 \end{equation}
 where $s$ is a positive number and $A$ is a non--negative function defined on $(0,p-1)$.

\begin{definition}[Generalized grand Morrey spaces]\label{def:ggms}
Let  $1<p<\infty$, $0\leqslant \lambda <1$, $\varphi$ be a positive bounded function with $\lim_{t \to 0+} \varphi(t)=0$ and $A$ be a non-decreasing real-valued non-negative function with $\lim_{x\to 0+} A(x)=0$. By
 $L^{p),\lb)}_{\varphi,A}(X,\mu)$ we denote the space of measurable functions having  the finite norm
\begin{equation}\label{equ:norm}
    \|f\|_{L^{p),\lb)}_{\varphi,A}(X)}:=
\Phi^{p,\lb}_{\varphi,A}(f,s_{\max}), \quad s_{\max}= \min\left\{p-1,a\right\},
\end{equation}
where $a=\sup \{x>0: A(x)\leqslant \lambda\}$.
\end{definition}

\begin{remark}\label{rem:quotient_morrey}
For appropriate $\varphi$, in the case $A\equiv 0, \lambda>0$ we recover  the Grand Morrey spaces introduced in  A. Meskhi \cite{meskhi}, and when  $\lambda=0$, $A\equiv 0$ we have the   grand Lebesgue spaces introduced in \cite{greco} (and in \cite{iwa_sbor1992} in the case $\theta=1$).
\end{remark}

For fixed $p, \lambda, \varphi, A, f$ we have that $ s \mapsto \Phi^{p,\lb}_{\varphi,A}(f,s)$ is a non-decreasing function, but  it is possible to  estimate $\Phi^{p,\lb}_{\varphi,A}(f,s)$ via $\Phi^{p,\lb}_{\varphi,A}(f,\sg)$ with $\sigma<s$ as follows.

\begin{lemma}\label{lem:dominance}  For $0<\sigma<s<s_{\max}$ we have that
\begin{equation}\label{equ:dominance}
    \Phi^{p,\lb}_{\varphi,A}(f,s) \leqslant C \varphi(\sigma)^{-\frac{1}{p-\sigma}} \Phi^{p,\lb}_{\varphi,A}(f,\sg),
\end{equation}
where $C$ depends on   $\gamma,$ the parameters $p,\lb, \varphi, A$  and the diameter $d_X,$ but does not depend on $f, s$ and $\sg$.
\end{lemma}

From Lemma \ref{lem:dominance} we immediately have
\begin{lemma}\label{lem:dominance1}
For $0<\sigma<s_{\max}$, the norm defined in \eqref{equ:norm} has the following dominant
\begin{equation}\label{equ:dominant}
    \|f\|_{L^{p),\lb)}_{\varphi,A}(X)}\leqslant C \frac{ \fun{p}{\lambda}{\varphi}{A}{\sigma}{f}}{\varphi(\sigma)^\frac{1}{p-\sigma}},
\end{equation}
where $C$ depends on   $\gamma,$ the parameters $p,\lb, \varphi, A$  and the diameter $d_X,$ but does not depend on $f$ and $\sg$.
\end{lemma}


\begin{lemma}[Extended reduction  lemma]\label{main}
Let $U$ and $\Lambda$ be operators (not necessarily sublinears) satisfying the following relation in Morrey spaces
    \begin{equation}\label{eq:boun_classical}
\|Uf\|_{L^{q-\ve,\lambda-A_2(\ve)}(X)} \leqslant    C_{p-\ve,\lambda-A_1(\ve),q-\ve,\lambda-A_2(\ve)} \|\Lambda f\|_{L^{p-\ve,\lambda-A_1(\ve)}(X)}
\end{equation}
for all sufficiently small $\ve\in (0,\sg]$, where
 $0<\sigma<s_{\max}.$ If
\begin{equation}\label{reduction_condition}
\sup_{0<\varepsilon< \sigma} C_{p-\ve,\lambda-A_1(\ve),q-\ve,\lambda-A_2(\ve)} <\infty
\end{equation}
and
\begin{equation}\label{eq:finite_ratio}
\sup_{0<\ve<\sigma} \frac{\psi(\ve)^{\frac{1}{q-\ve}}}{\varphi(\ve)^{\frac{1}{p-\ve}}}<\infty,
\end{equation}
then the relation is also valid in the generalized grand Morrey space
\begin{equation}\label{equ:metatheorem}
\|Uf\|_{L^{q),\lb)}_{\psi,A_2}(X)} \leqslant C\|\Lambda f\|_{L^{p),\lb)}_{\varphi,A_1}(X)}
\end{equation}
with
\[
 C=\frac{C_0}{\varphi(\sigma)^\frac{1}{p-\sigma}} \sup_{0<\varepsilon< \sigma} C_{p-\ve,\lambda-A_1(\ve),q-\ve,\lambda-A_2(\ve)},
\]
where $C_0$ may depend on $\gamma,p,\lb,\varphi, A$ and  $d_X,$ but does not depend on $\sg$ and $f$.
\end{lemma}
\begin{proof}
The proof follows the same lines as for the case where $\Lambda$
is the identity operator, see \cite{kokmesraf},
\cite{kokmesrafarxiv} for that proof.
\end{proof}

Using the reduction lemma we obtain the boundedness of maximal and Calder\'on-Zygmund operators in generalized grand Morrey spaces, namely.

\begin{theorem}\label{theo:moggMS}
Let $1<p<\infty$ and $0\leqslant \lambda <1$. Then the Hardy-Littlewood maximal operator is
bounded from $L^{p),\lambda)}_{\varphi,A}(X, \mu)$ to $L^{p),\lambda)}_{\psi,A}(X,\mu)$ if there exists small
$\sg$ such that
$
\sup_{0<\ve<\sigma} \psi(\ve)^{\frac{1}{q-\ve}}/\varphi(\ve)^{\frac{1}{p-\ve}}<\infty.
$
\end{theorem}

\begin{theorem}\label{theo:CZggMs}Let   $1<p<\infty$, $\theta >0$ and let $0<\lambda<1$.
Then the Calder\'on-Zygmund operator $T$ is bounded in the generalized grand Morrey spaces $L^{p),\lambda)}_{\theta,
A}(X,\mu)$.
\end{theorem}

\section{Boundedness of Commutators in Generalized Grand Morrey spaces}

\subsection{Commutator of Calder\'on-Zygmund operators}
Before proving the main result in this subsection, we need some auxiliary results.  The following lemma was proved in \cite{ye_xf} but we give the proof for completeness of presentation.
\begin{lemma}\label{lem:2.6}
Let $T$ be a Cald\'eron-Zygmund operator, $ 0<s<\infty$, if $b\in \bmo $,  then there exist a constant $C>0$ such that for all functions $f$  with compact support,
\begin{equation}\label{eq:pointwisecommutatorriesz}
([b,T]f)^\sharp(x)\leqslant C \norm{b}_{\bmo}\left(M_{s}(Tf)(x)+M_{s}(f)(x)\right).
\end{equation}
\end{lemma}
\begin{proof}
For any ball $B=B(x, r)\subset X$,  we write

\[
\begin{split}
\com{b}{T}f(y)&=\com{b-b_{2B}}{ T}f(y)\\
&=(b-b_{2B})Tf(y)-T((b-b_{2B})f\chi_{2B})(y)-T\left((b-b_{2B})f\chi_{(2B)^{\complement}}\right)(y)\\
&=\mathcal I_1(y)-\mathcal I_2(y)-\mathcal I_3(y).
\end{split}
\]
Then, we obtain
\[
\begin{split}
\fint_{B}|\com{b}{ T}f(y)-\mathcal I_3(z)| \dif \mu(y)&\leqslant\fint_{B}|b(y)-b_{2B}||Tf(y)|\dif \mu(y) +\fint_{B}|T((b-b_{2B})f\chi_{2B})(y)|\dif \mu(y)\\
&+\fint_{B}\left|T\left((b-b_{2B})f\chi_{(2B)^{\complement}}\right)(y)-T\left((b-b_{2B})f\chi_{(2B)^{\complement}}\right)(z)\right|\dif \mu(y)\\
&:=\mathscr{I}_{1}(x)+\mathscr{I}_{2}(x)+\mathscr{I}_{3}(x,z).
\end{split}
\]

The estimation $ \mathscr{I}_{1}(x) \leqslant C\norm{b}_{\bmo}M_{s}(Tf(x)) $  is obtained via H\"older's inequality.\\

To estimate $\mathscr{I}_2$, there exists $s_0,s_1>1$, such that $1/s_{0}=1/s_{1}+1/s$, by H\"older's inequality, the $L^s$ boundedness of $T$ and the Remark \ref{rem:2.1} we have
\[
\begin{split}
\mathscr{I}_{2}(x)&\leqslant \left(\fint_{B}|T((b-b_{2B})f\chi_{2B})(y)|^{s_{0}}\dif \mu(y)\right)^{1/s_{0}}\\
&\lesssim \left(\fint_{2B}|b(y)-b_{2B}|^{s_{0}}|f(y)|^{s_{0}}\dif \mu(y)\right)^{1/s_{0}}\\
&\lesssim \left(\fint_{B}|b(y)-b_{2B}|^{s_{1}}\dif\mu(y)\right)^{1/s_{1}}\left(\fint_{B}|f(y)|^{s}\dif\mu(y)\right)^{1/s}\\
&\lesssim \norm{b}_{\bmo}M_{s}(f(x)).
\end{split}
\]
 For any $z\in B$ and $ y\in(2B)^{\complement}$,  $2d(z, x)\leqslant d(y, x)$; it follows from (ii) of the definition of  Calder\'on-Zygmund operator that
\[
\begin{split}
\mathscr{I}_{3}(x) &\lesssim \fint_{B}\int_{(2B)^{c}}|K(z, y)-K(x, y)||b(y)-b_{2B}||f(y)|\dif \mu(y)\dif \mu(z)\\
&\lesssim  \fint_{B}\int_{(2B)^{\complement}}w\left(\frac{d(z,x)}{d(y,x)}\right)\frac{1}{\mu B(x,d(x,y))}|b(y)-b_{2B}||f(y)|\dif \mu(y)\dif \mu(z)\\
&\lesssim \sum_{k=1}^{\infty}w\left(2^{-k}\right)\frac{1}{\mu B(x,2^{k}r)}\int_{2^{k+1}B}|b(y)-b_{2B}||f(y)|\dif \mu(y)\\
&\lesssim \sum_{k=1}^{\infty}w\left(2^{-k}\right)\left[\frac{1}{\mu B(x,2^{k+1}r)}\int_{2^{k+1}B}|b(y)-b_{2B}|^{s'}\dif \mu(y)\right]^{\frac{1}{s'}}\left[\frac{1}{\mu B(x,2^{k+1}r)}\int_{2^{k+1}B}|f(y)|^{s}\dif\mu(y)\right]^{\frac{1}{s}}\\
&\lesssim  \norm{b}_{\bmo}M_{s}(f)(x)\sum_{k=1}^{\infty} w\left(2^{-k}\right),
\end{split}
\]
since $w(t)$  is a positive non-decreasing function on $(0, \infty)$ and satisfies the Dini condition,
\[
\sum_{k=1}^{\infty}w(2^{-k})\leqslant c\int_{0}^{1}\frac{w(t)}{t}\dif t <\infty.
\]
Therefore, we have $\mathscr{I}_{3}(x)\lesssim \norm{b}_{\bmo}M_{s}(f)(x)$. Thus
\[
\begin{split}
([b,\ T]f)^{\sharp}(x)&=\sup_{0<r<d}\inf_{a\in R}\fint_{B(x,r)}|[b,\ T]f(y)-a|^{s}\dif \mu(y)\\
&\lesssim  \norm{b}_{\bmo}(M_{s}(Tf)(x)+M_{s}(f)(x)). \qedhere
\end{split}
\]
\end{proof}

\begin{lemma}\label{lemma:sharp_maximal_estimates}
Let $1<p<\infty$, $0<\lambda<1$, $\varphi$ and $A$ as in the definition of generalized grand Morrey spaces.    If $b \in \bmo$, then we have
\begin{enumerate}[(i)]
\item $\norm{M\left(\com{b}{T}f \right)}_{\ggMs{p}{\lambda}{\theta}{A}{X}}\lesssim \norm{\left(\com{b}{T}f \right)^{\sharp}}_{\ggMs{p}{\lambda}{\theta}{A}{X}};$
\item $\norm{\left(\com{b}{T}f \right)^{\sharp}}_{\ggMs{p}{\lambda}{\theta}{A}{X}}\lesssim \norm{b}_{\bmo}\left (  \norm{M_s(Tf)}_{\ggMs{p}{\lambda}{\theta}{A}{X}} +    \norm{M_sf}_{\ggMs{p}{\lambda}{\theta}{A}{X}} \right),$
\end{enumerate}
where $M$ is the maximal operator and $T$ is the Calder\'on-Zygmund operator.
\end{lemma}

\begin{proof}
The result of (i) follows at once taking into account Lemmas \ref{lemma:feffermanstein} and \ref{main}. For (ii), we simply use the pointwise estimate \eqref{eq:pointwisecommutatorriesz}.
\end{proof}

\begin{theorem}\label{theo:commutator_CZ}
Let $ 1<p<\infty$, $\theta>0$ and $0<\lambda<1$.  Suppose $T$  is a Calder\'on-Zygmund operator and $b \in \bmo$. Then the commutator $\com{b}{T}$  is bounded in $\ggMs{p}{\lambda}{\theta}{A}{X}$.
\end{theorem}

\begin{proof}
For any $ 1<p<\infty$, $\theta>0$ and $0\leqslant\lambda<1$,  there exists an $s$ such that $ 1<s<p<\infty$ and sufficiently  small positive number $\sigma$. By above lemmas, we have
\[ \begin{split}
\norm{\com{b}{T}f}_{\ggMs{p}{\lambda}{\theta}{A}{X}}&\leqslant \norm{M (\com{b}{T}f)}_{\ggMs{p}{\lambda}{\theta}{A}{X}}\\
&\leqslant C\sup_{0<\ve\leqslant\sigma} \left(b^{(\lambda -A(\varepsilon))/(p-\varepsilon)}+1\right)\norm{([b,\ T]f)^{\#}}_{\ggMs{p}{\lambda}{\theta}{A}{X}}\\
&\leqslant C\left(\norm{M_{s}(Tf)}_{\ggMs{p}{\lambda}{\theta}{A}{X}}+\norm{M_{s}f}_{\ggMs{p}{\lambda}{\theta}{A}{X}}\right)\\
&\leqslant C\sup_{0<\ve\leqslant\sigma}\left(b^{\lambda s/(p-\ve)}\left(\left(\frac{p-\ve }{s}\right)'\right)^{\frac{\mathrm{s}}{p-\in}}+1\right) \left(\norm{Tf}_{\ggMs{p}{\lambda}{\theta}{A}{X}}+\norm{f}_{\ggMs{p}{\lambda}{\theta}{A}{X}}\right)\\
&\leqslant C\left(\norm{Tf}_{\ggMs{p}{\lambda}{\theta}{A}{X}}+\norm{f}_{\ggMs{p}{\lambda}{\theta}{A}{X}}\right)\\
&\leqslant C \norm{f}_{\ggMs{p}{\lambda}{\theta}{A}{X}}.
\end{split}
\]
\end{proof}

\subsection{Commutators of potential operators}

Let $0<\alpha<1$ and let 
$$ I^{\alpha}f(x)=\int_{X}\frac{f(y)}{\mu B(x,d(x,y))^{1-\alpha}}\dif \mu(y) $$
be the potential operator. 

The following lemma was shown in \cite{ye_xf} which follows from well-known arguments; we give the proof with slight modification for completeness of presentation and for  convenience of the reader.

\begin{lemma}\label{lem:3.2}
 Let $I^{\alpha}$ be a potential operator, $1<p<\infty,\ 0< \alpha<(1-\lambda)/p$, $0\leqslant\lambda<1$ and
$1/p-1/q=\alpha/(1-\lambda)$.  If $b\in \bmo$,  then there exists a constant $C_{p, \alpha, \lambda}>0$ such that for all functions $f$  with compact support,
\begin{equation}\label{eq:maximal_comutador}
\norm{ M(\com{b}{I^{\alpha}}f)}_{\Ms{q}{\lambda}}\leqslant C_{p,q, \alpha, \lambda}\norm{b}_{\bmo}\norm{f}_{\Ms{p}{\lambda}},
\end{equation}
where
\begin{equation}\label{equ:constante_maximal_comutador}
C_{p,q, \alpha, \lambda}=C\left(b^{\lambda s/p}((p/s)')^{\frac{\mathrm{s}}{p}}+1\right)^{1+\frac{p}{q}}\left(1+\frac{p}{1-\lambda-\alpha p}\right)[(p')^{1/q}+1].
\end{equation}
\end{lemma}
\begin{proof} For any ball $B=B(x,r)\subset X$ and  any real number $c$,  we write
\[
\begin{split}
\com{b}{I_{\alpha}}f(y)&=\com{b-c}{I_{\alpha}}f(y)\\
&=(b-c)I^{\alpha}f(y)-I^{\alpha}((b-c)f\chi_{c_0B})(y)-I^{\alpha}\left((b-c)f\chi_{(c_0B)^{\complement}}\right)(y)\\
&=\mathscr I_1(y)-\mathscr I_2(y)-\mathscr I_3(y),
\end{split}
\]
where $c_0$ is the constant depending on $C_t$ and $C_s$, and will be determined later. 
Then, by the sublinearity of the maximal operator, we have
\[
M(\com{b}{I_{\alpha}}f)(x)\leqslant M\mathscr I_1(x)+M\mathscr I_2(x)+M\mathscr I_3(x).
\]

For $M\mathscr I_{1}(x)$, we have the pointwise estimate $M\mathscr I_{1}(x) \leqslant C\norm{b}_{\bmo}M_{s}(I^{\alpha}f(x))$ which follows from  H\"{o}lder's inequality. Taking  the boundedness of $M_s$ and $I^\alpha$ into account, we have
\begin{equation}\label{eq:M1}
\begin{split}
\norm{M \mathscr I_{1}}_{\Ms{q}{\lambda}}&\lesssim \norm{b}_{\bmo}\norm{M_{s}(I^{\alpha}f)}_{\Ms{q}{\lambda}}\\
&\lesssim(b^{\lambda s/p}((p/s)')^{\frac{\mathrm{s}}{p}}+1)\norm{b}_{\bmo}\norm{I^{\alpha}f}_{\Ms{q}{\lambda}}\\
&\lesssim C_{p,\alpha,\lambda}\norm{b}_{\bmo}\norm{f}_{\Ms{p}{\lambda}},\\
\end{split}
\end{equation}
where $C_{p,\alpha,\lambda}$ is \eqref{equ:constante_maximal_comutador}.

For  $1<s<p<\infty$, $0<\alpha< (1-\lambda)/p$, there exists $s_{0}, s_{1}, t_{0}, t>1$, such  that $1/s_{0}=  1/t_{0}-\alpha$, $1/t_{0}=1/s_{1}+1/s$ and $s/t=(\alpha p)/(1-\lambda)$. By H\"{o}lder's  inequality together with Jensen's inequality and the fact that $I^\alpha$ is of strong type $(t_{0}, s_{0})$ we have (remembering that $B:=B(x,r))$

\begin{equation}\label{eq:M2}
\begin{split}
M &\mathscr I_{2}(x)  \\
&\leq  \sup_{0< r<d_X}\left(\fint_{B}|I_{\alpha}((b-c)f\chi_{c_0B})(y)|^{s_{0}}\dif \mu(y)\right)^{1/s_{0}}\\
&\lesssim\sup_{0< r<d_X}\left(\frac{1}{\mu(B)^{1-t_{0}\alpha}}\int_{c_0B}|b(y)-c|^{t_{0}}|f(y)|^{t_{0}}\dif\mu(y)\right)^{1/t_{0}}\\
&\lesssim\sup_{0< r<d_X}\left(\fint_{c_0B}|b(y)-c|^{s_{1}}\dif\mu(y)\right)^{1/s_{1}}\left(\frac{1}{\mu(c_0B)^{1-s\alpha}}\int_{c_0B}|f(y)|^{s}\dif\mu(y)\right)^{1/s}\\
&\lesssim\norm{b}_{\bmo}\sup_{0< r<d_X}\mu(c_0B)^{\alpha-\frac{1}{s}}\left(\int_{c_0B}|f(y)|^{s}\dif\mu(y)\right)^{\frac{1}{s}-\frac{1}{t}} \left[\mu(c_0B)^{1-\frac{s}{p}}\left(\int_{c_0B}|f(y)|^{p}d\mu(y)\right)^{s/p}\right]^{\frac{1}{t}}\\
&\lesssim\norm{b}_{\bmo}\sup_{0< r<d_X}\mu(c_0B)^{\alpha-\frac{1}{s}+\frac{1}{t}-\frac{s(1-\lambda)}{pt}}\left(\int_{c_0B}|f(y)|^{s}\dif\mu(y)\right)^{\frac{1}{s}-\frac{1}{t}} \left(\frac{1}{\mu(c_0B)^{\lambda}}\int_{c_0B}|f(y)|^{p}\dif \mu(y)\right)^{\frac{\mathrm{s}}{pt}}\\
&\lesssim\norm{b}_{\bmo}\norm{f}^{\frac{\alpha p}{1-\lambda}}_{\Ms{p}{\lambda}} \sup_{0<r<d_X}\left(\fint_{c_0B}|f(y)|^{s}\dif\mu(y)\right)^{\frac{1}{s}(1-\frac{s}{t})}\\
&\lesssim\norm{b}_{\bmo}\norm{f}^{\frac{\alpha p}{1-\lambda}}_{\Ms{p}{\lambda}}(M_{s}f(x))^{1-\frac{\alpha p}{1-\lambda}}.
\end{split}
\end{equation}
Consequently, by Lemma \ref{lem:2.2},
\begin{equation}\label{eq:MI2}
\begin{split}
\norm {M\mathscr I_{2}}_{\Ms{q}{\lambda}}&\lesssim\norm{b}_{\bmo}  \norm{f}^{\frac{\alpha p}{1-\lambda}}_{\Ms{p}{\lambda}}\norm{(M_{s}f)^{1-\frac{\alpha p}{1-\lambda}}}_{\Ms{q}{\lambda}}\\
&\lesssim\norm{b}_{\bmo} \norm{f}^{\frac{\alpha p}{1-\lambda}}_{\Ms{p}{\lambda}} \norm{M_{s}f}_{L^{p,\lambda}}^{p/q}\\
&\lesssim (C_{p,\alpha,\lambda})^{p/q}\norm{b}_{\bmo}\norm{f}_{\Ms{p}{\lambda}},
\end{split}
\end{equation}
where $C_{p,\alpha,\lambda}$ is \eqref{equ:constante_maximal_comutador}.

Since we have the validity of the reverse doubling condition, see
\eqref{eq:rd}, there exists  constants  $0<\alpha$, $\beta<1$ such
that for all $x\in X$ and small positive $r$, $ \mu B(x,\ \alpha
r)\leqslant\beta\mu B(x,r).$ Let us  take an  integer $m$ so that
$\alpha^{m}d_X$  is sufficiently  small.

Observe now that (see also \cite[p. 929]{KMCVEE}) if  $z\in B(x,r)$,  then $B(x, r)\subset B(z,C_t(C_s+1)r)\subset
B\left(x, C_t(C_t(C_s+1)+1)r\right)$ (rewrite it simply as $B(x,r)\subset B(z,c_1r)\subset B(x,c_2r)$). Hence,
\[
\begin{split}
\norm {M \mathscr I_{3}}_{\Ms{q}{\lambda} } &\leqslant \sup_{\substack{x\in X \\0 <r<d_X}} \left(\frac{1}{\mu B(x,r)^{\lambda}}\int_{B(x,r)}|M(\mathscr I_{3})(y)|^{q}\dif\mu(y)\right)^{\frac{1}{q}}\\
&\lesssim  \sup_{\substack{x\in X \\0 < r<d_X}}  \mu B(x, r)^{\frac{1-\lambda}{q}}\sup_{B\subset B(z,c_{1}r)}\frac{1}{\mu B(z,c_{1}r)}\int_{B(z,c_{1}r)}|\mathscr I_{3}(y)|\dif \mu(y)\\
&\lesssim \sup_{B\subset B(z,c_{1}r)}\mu B(z,\ c_{1}r)^{\frac{1-\lambda}{q}-1}\int_{B(z,c_{1}r)}|\mathscr I_{3}(y)|\dif\mu(y).
\end{split}
\]
Further, notice  that when $c_0$ is an appropriate constant,  $B\subset B(z, c_1r)$, $y\in B(z, c_1r)$, $\alpha^m  d(y,z) \leq d(y,t) \leq \alpha^{m+1}d(y,z)$ and $z\in (c_0 B)^c$, then $d(x,t)>\bar{c}_0 r$, where $\bar{c}_0$ depends on $C_t$, $C_s$ and $c_0$; it is also  easy to see check that there are  positive constants $b_1$, $b_2$ and $b_3$ such that
$B(y, b_1d(y,t)) \subset B(x, b_2 d(x,t))\subset B(y, b_3 d(y,t))$. Consequently, by using  Fubini's  theorem, estimates of integrals (see Lemma 1.2 in \cite{317zd})   we have for $y\in B(z,c_1r)$,

\[
\begin{split}
\mathscr I_{3}(y) &\leqslant \int_{X\backslash c_0B}|(b(z)-c)f(z)| \mu B(y,d(y, z))^{\alpha-1}\dif\mu(z)\\
&\leqslant C \int_{X\backslash c_0 B}|(b(z)-c)f(z)| \left( \int_{B(y,\alpha^{m}d(y,z))\backslash B(y,\alpha^{m-1}d(y,z))}\mu B(y, d(y, t))^{\alpha-2}\dif\mu(t)\right) \dif\mu(z)\\
&\leqslant C\int_{X\backslash B\left( x,\bar{c}_0 r\right)}\mu B(y,d(y, t))^{\alpha-2}\left(\int_{B \left(y,a^{-m}d(y,t)\right)}|(b(z)-c)f(z)|\dif\mu(z)\right)\dif\mu(t)\\
&\leqslant C\norm{b}_{\bmo}  \int_{X\backslash B\left( x,\bar{c}_0 r\right)}  \mu B(y, d(y, t))^{\alpha-1} \times  \\
& \hspace{6cm} \left(\frac{1}{\mu B(y,a^{1-m}d(y,t))}\int_{B(y,a^{1-m}d(y,t))}|f(z)|^{p}\dif\mu(z)\right)^{\frac{1}{p}}\dif\mu(t)\\
 &\leqslant C\norm{b}_{\bmo}\norm{f}_{\Ms{p}{\lambda}}\int_{X\backslash B\left( x,\bar{c}_0 r\right)} \mu B(y,d(y,t))^{\alpha-\frac{1-\lambda}{p}-1}\dif\mu(t)\\
 &\leqslant C\mu B(x,r)^{\alpha-\frac{1-\lambda}{p}}\norm{b}_{\bmo}\norm{f}_{\Ms{p}{\lambda}}.
\end{split}
\]
Thus applying the relation between $B(x,r)$ and $B(z,r)$ we find
that
\begin{equation}\label{eq:MI3}
\begin{split}
 \norm{M\mathscr I_{3}}_{\Ms{q}{\lambda}}& \lesssim \norm{b}_{\bmo}\norm{f}_{\Ms{p}{\lambda}}\sup_{B\subset B(z,c_{1}r)}\mu B(z,\ c_{1}r)^{\frac{1-\lambda}{q}+\alpha-\frac{1-\lambda}{p}}\\
 &\lesssim \norm{b}_{\bmo}\norm{f}_{\Ms{p}{\lambda}}.
\end{split}
\end{equation}
Gathering \eqref{eq:M1}, \eqref{eq:MI2}, \eqref{eq:MI3} it is easy to show that
\[
\norm{M(\com{b}{I^{\alpha}}f)}_{\Ms{q}{\lambda}}\leqslant C_{p,q, \alpha, \lambda}\norm{b}_{\bmo}\norm{f}_{\Ms{p}{\lambda}}.\qedhere
\]
\end{proof}

Before proving the next result, we define the following auxiliary functions which where introduced in \cite{kokmesraf}.

\begin{definition}[auxiliary functions]\label{eq:auxiliary_functions}
On an interval $(0,\delta ]$, $\delta $ is small,  we define the following
functions:
\[
\bar{\phi}(x):=p+\frac{(x-q)(1-\lambda +A_{2}(x))}{1-\lambda +A_{2}(x)-\alpha (x-q)},\; \tilde{\phi}(x):=q-\frac{(p-x)(1-\lambda +A_{1}(x))}{1-\lambda +A_{1}(x)-\alpha (p-x)}
\]

\[
\bar{A}(x)=1-\frac{\alpha(x-q)}{1-\lambda +A_{2}(x)}, \; \tilde{A}(x)=\frac{%
1-\lambda +A_{1}(\eta )}{1-\lambda +A_{1}(\eta)-(p-\eta )\alpha }
\]

\[
\phi (x):= \bar{\phi}(x)^{\bar{A}(x)},\quad
\Phi (x):= \tilde{\phi}(x)^{\tilde{A}(x)}
\]

\[
\psi (\varepsilon )=\phi (\varepsilon ^{\theta _{1}}), \quad
\text{ }\Psi (\varepsilon )=\Phi (\varepsilon ^{\theta _{1}}),
\]
for $\theta_{1}>0$.
\end{definition}

\begin{theorem}\label{theo:4.3}

 Let $I^{\alpha}$ be a potential operator and let $M$ be the maximal operator. Assume that $1<p<\infty,$
 $0<\alpha <(1-\lambda)/p,$ $0<\lambda <1,$  $1/p-1/q=\alpha/(1-\lambda).$ Suppose that $\theta _{1}>0$
 and that $\theta _{2}\geqslant\theta
_{1}[1+\alpha q/(1-\lambda)].$ Let $A_{1}$ and $A_{2}$ be continuous
non-negative functions on $(0,p-1]$ and $(0,q-1]$ respectively satisfying
the conditions:
\begin{itemize}
\item[(i)] $A_{2}\in C^{1}((0,\delta ])$ for some positive $\delta >0;$
\item[(ii)] $\lim_{x\rightarrow 0+}A_{2}(x)=0;$
\item[(iii)] $0\leqslant B:=\lim_{x\rightarrow 0+} \frac{\dif }{\dif x}A_2(x)<\frac{%
(1-\lambda )^{2}}{\alpha q^{2}};$
\item[(iv)] $A_{1}(\eta )=A_{2}(\bar{\phi}^{-1}(\eta )),$ where $\bar{\phi}%
^{-1}$ is the inverse of $\bar{\phi}$ on $(0,\delta ]$ for some $\delta >0.$
\end{itemize}

If $b\in \bmo$,  then the operator $M(\com{b}{I^\alpha})$ is bounded from
$\ggMs{p}{\lambda}{\theta_1}{A_1}{X}$ to $\ggMs{q}{\lambda}{\theta_2}{A_2}{X}$.

\end{theorem}

\begin{proof}
We note that it is enough to prove the theorem for $\theta _{2}=\theta _{1}(1+\frac{%
\alpha q}{1-\lambda })$ because $\varepsilon ^{\theta _{2}}\leqslant \varepsilon
^{\theta _{1}(1+\frac{\alpha q}{1-\lambda })}$ for $\theta _{2}>\theta
_{1}[1+(\alpha q)/(1-\lambda)]$ and small $\ve$. We also note that, by L'Hospital rule,  $\bar{\phi}(x)\sim x$ as $\ x\rightarrow 0+$ since $B<(1-\lambda )^2/(\alpha q^2)$. Moreover, $\bar{\phi}$ is invertible near $0$, since $\frac{d \bar{\phi}}{dx}(x)>0.$ Under the conditions of Theorem \ref{theo:4.3} the function $A_{1}$ is continuous on $(0,\delta ]$ and $\lim_{x\rightarrow0+}A_{1} (x)=0.$ With all of the previous remarks taken into account, it is enough to prove the boundedness of $M(\com{b}{I^\alpha})$ from $ L_{\theta _{1},A_{1}}^{p),\lambda)}(X,\mu)$
to $ L_{\psi, A_{2}}^{q),\lambda)}(X,\mu)$ since $\phi(x) \sim  x^{1+\frac{\alpha q}{%
1-\lambda }}$, and consequently, $\psi (x)=\phi (x^{\theta _{1}})\sim
x^{\theta _{1}\left(1+\frac{\alpha q}{1-\lambda }\right)}$ as $x\rightarrow 0. $

The case $\sigma<\varepsilon\leqslant s_{\max}$, where $s_{\max}$ is  from \eqref{equ:norm}. Letting
\[
I:=\psi^{\frac{1}{q-\varepsilon}}(\varepsilon)\left(\frac{1}{\mu B(x,r)^{\lambda-A_2(\varepsilon)}} \int_{B(x,r)} |M(\com{b}{I^\alpha} f)(y)|^{q-\varepsilon} \dif \mu(y) \right)^\frac{1}{q-\varepsilon}
\]
we have
\[
\begin{split}
I&\lesssim \psi^{\frac{1}{q-\varepsilon}}(\varepsilon) \mu B(x,r)^\frac{A_2(\varepsilon)+1-\lambda}{q-\varepsilon}\left( \fint_{B(x,r)} | M(\com{b}{I^\alpha} f)(y) |^{q-\sigma} \dif \mu(y) \right)^\frac{1}{q-\sigma}\\
&\lesssim  \psi^{\frac{1}{q-\varepsilon}}(\varepsilon) \mu B(x,r)^\frac{A_2(\sigma)+1-\lambda}{q-\sigma}\left( \fint_{B(x,r)} | M(\com{b}{I^\alpha} f)(y)  |^{q-\sigma} \dif \mu(y) \right)^\frac{1}{q-\sigma}\\
&\lesssim  \left (\sup_{\sigma\leqslant \varepsilon\leqslant s_{\max}}\psi^{\frac{1}{q-\varepsilon}}(\varepsilon)\right ) \psi^\frac{1}{\sigma-q}(\sigma)
\sup_{0<\varepsilon\leqslant \sigma} \sup_{\stackrel{x\in X}{r>0}} \left(\frac{\psi(\varepsilon)}{\mu B(x,r)^{\lambda-A_2(\varepsilon)}} \int_{B(x,r)} | M(\com{b}{I^\alpha} f)(y)  |^{q-\varepsilon} \dif \mu(y) \right)^\frac{1}{q-\varepsilon},
\end{split}
\]
where the first inequality comes from H\"older's inequality and the second one is due to the fact that $A_2$ is bounded on $[\sigma,q-1)$ and $x\mapsto (1-\lambda)/(q-x)$ is an increasing function. Hence, it is enough to consider the case $0<\varepsilon\leqslant \sigma $.\par
The case $0<\varepsilon\leqslant \sigma$. Let $\eta$ and $\varepsilon$ be chosen so that
\begin{equation}\label{eq:etavar}
\frac{1}{p-\eta}-\frac{1}{q-\varepsilon}=\frac{\alpha}{1-\lambda+A_2(\varepsilon)}.
\end{equation}
Obviously we have that $\varepsilon \to 0$ if and only if  $\eta \to 0$ and solving $\eta$ with  respect to $\varepsilon$ in \eqref{eq:etavar} we obtain
\[
\eta=p-\frac{(q-\varepsilon)(1-\lambda+A_2(\varepsilon))}{1-\lambda+A_2(\varepsilon)-\alpha(\varepsilon-q)}= \bar \phi(\varepsilon).
\]
Letting
\[
J:=\psi^{\frac{1}{q-\varepsilon}}(\varepsilon)\left(\frac{1}{\mu B(x,r)^{\lambda-A_2(\varepsilon)}} \int_{B(x,r)} | M(\com{b}{I^\alpha} f)(y) |^{q-\varepsilon} \dif \mu(y) \right)^\frac{1}{q-\varepsilon}
\]
we have
\[
\begin{split}
J&\lesssim C_{p-\eta,q-\ve,\alpha,\lambda-A_2(\ve)} \; \psi^\frac{1}{q-\varepsilon}(\varepsilon) \sup_{\stackrel{x\in X}{r>0}}\left( \frac{1}{\mu B(x,r)^{\lambda-A_2(\varepsilon)}} \int_{B(x,r)} |f(y)|^{p-\eta} \dif \mu(y) \right)^\frac{1}{p-\eta} \\
&\lesssim  C_{p-\eta,q-\ve,\alpha,\lambda-A_2(\ve)} \; \eta^\frac{\theta_1}{\eta-p}\psi^\frac{1}{q-\varepsilon}(\varepsilon) \sup_{\stackrel{x\in X}{r>0}}\left( \frac{\eta^{\theta_1}}{\mu B(x,r)^{\lambda-A_2(\varepsilon)}} \int_{B(x,r)} |f(y)|^{p-\eta} \dif \mu(y) \right)^\frac{1}{p-\eta} \\
&\lesssim \|f\|_{ L^{p),\lambda)}_{\theta_1,A_1}(X,\mu)},
\end{split}
\]
where the first inequality is due to Lemma \ref{lem:3.2} and the constant $C_{p-\eta,q-\varepsilon,\alpha,\lambda-A_2(\varepsilon)}$ is the one from \eqref{equ:constante_maximal_comutador}. The last inequality is due to the fact that $\eta=\bar \phi(\varepsilon)$. Since the constant in the last inequality is uniformly bounded with respect  to  $\varepsilon$  we obtain the desired boundedness of the operator.
\end{proof}


\begin{corollary}
Let the conditions of Theorem \ref{theo:4.3} be fullfiled.
Then the commutator $[b,\ I_{\alpha}]$  is bounded from
$\ggMs{p}{\lambda}{\theta_1}{A_1}{X}$   to  $\ggMs{q}{\lambda}{\theta_2}{A_2}{X}$.
\end{corollary}

\begin{proof}
The result follows by the previous theorem and by the inequality

\[\norm{\com{b}{I^{\alpha}}f}_{\ggMs{q}{\lambda}{\theta_2}{A_2}{X}} \leqslant \norm{M(\com{b}{I^{\alpha}}f)}_{\ggMs{q}{\lambda}{\theta_2}{A_2}{X}}.\qedhere \]
\end{proof}

\section{Interior estimate of elliptic equations}

In this section we apply the main result of this paper to
establish some interior estimates of solutions to nondivergence
elliptic equations with $VMO$ coefficients. (see also the paper
\cite{liu} for related topics). Suppose $n \geq 3$ and $\Omega$ is
an open set in $\mathbb R^{n}$. Let
\[
Lu=\sum_{i,j=1}^{n}a_{ij}(x)(\partial^{2}/\partial x_{i}\partial x_{j}),
\]
where $a_{ij}=a_{ji}$ for $i,\ j=1,2,\ \cdots,\ n$, a.e. in $\Omega$; assume that there exists $C>0$ such that, for $y=(y_{1},\ \cdots,\ y_{n})\in \mathbb R^{n}$,
\begin{center}
$C^{-1}|y|^{2}\displaystyle \leqslant\sum_{i,j=1}^{n}a_{ij}(x)y_{i}y_{j}\leqslant C|y|^{2}$, for a.e. $ x\in\Omega$;
\end{center}
denote by $(A_{ij})_{n\times n}$ the inverse of the matrix $(a_{ij})_{n\times n}$. For $ x\in\Omega$ and $y\in \mathbb R^{n}$, let
\[
K(x,\ y)=\frac{1}{\left[(n-2)C_{n}\sqrt{\det(a_{ij}(x))}\right]}\left(\sum_{i,j=1}^{n}A_{ij}(x)y_{i}y_{j}\right)^{1-n/2}
\]
and $K_{i}(x,\displaystyle \ y)=\frac{\partial}{\partial y_{i}}K(x,\ y),\ K_{ij}(x,\displaystyle \ y)=\frac{\partial^{2}}{\partial x_{i}\partial x_{j}}K(x,\ y)$.

We denote by $VMO(\Omega)$ the class of all locally integrable functions with vanishing mean oscillation introduced
in \cite{Sar}. From \cite{[2],[6]}, we obtain the \textit{interior representation formula}, that is, if $a_{ij}\in VMO\cap L^{\infty}(\Omega)$ and $u\in W_{0}^{2,r}(\Omega)$, $1<r<\infty$ (see \cite{[2]}, \cite{[2']},\cite{[8]}),
\begin{multline*}
u_{x_{i}x_{j}}(x)=\mathrm{P}.\mathrm{V}.\ \int_{B}K_{ij}(x, x-y)\left[\sum_{k,l=1}^{n}(a_{kl}(x)-a_{kl}(y))u_{x_{k}x_{l}}(y)+Lu(y)\right]\dif y \\
+Lu(x)\displaystyle \int_{|y|=1}K_{i}(x, y)y_{j}\dif \delta_{y},
\end{multline*}
a.e. for $x\in B\subset\Omega$, where $B$ is a ball in $\Omega$. We also set
\[
M:= \max_{i,j=1,\ldots,n} \max_{|\alpha|\leqslant 2n} \norm{\partial^{\alpha}K_{ij}(x,y)/\partial y^{\alpha}}_{L^{\infty}}.
\]
To prove the next statement we need local version of Theorem 4.3 (see also Theorem 2.4 in \cite{[2]} or Theorem 2.13 in \cite{[2']}).
\vskip+0.2cm

{\bf Corollary A.} {\em Let $1<p<\infty$ and let $\Omega$ be a bounded domain in ${\Bbb{R}}^n$. Suppose that $a\in VMO\cap L^{\infty}$. Assume that $T$ is the Calder\'on--Zygmund operator defined  on $\Omega$ and that $\eta$ is the $VMO$ modulus of  $a$.  Then for any $\varepsilon>0$, there exists a positive number $\rho= \rho( \varepsilon, \eta)$ such that for any balls $B_r$ with the conditions: $\Omega_r:=B_r\cap \Omega\neq \emptyset$,  $r\in (0, \rho)$ and all $f\in L^{p), \lambda}_{\theta, A}(\Omega_r)$ the inequality
$$ \| [a, T]f\|_{L^{p), \lambda}_{\theta, A}(\Omega_r)} \leq C\varepsilon \|f\|_{L^{p), \lambda}_{\theta, A}(\Omega_r)}$$
is fulfilled.}

\begin{theorem}
Let $\Omega$ be a bounded domain in $\mathbb R^{n}$. Suppose that $1<p,r<\infty$. Let  $ a_{ij}\in  \vmo{\Omega}\cap L^{\infty}$, $i,\ j=1,2,\ \cdots,\ n$. Suppose that $\eta_{i,j}$ is the $VMO$ modulus of  $a_{ij}$; we set
$\eta=\Big( \sum_{i,j=1}^n \eta_{i,j}\Big)^{1/2}$. Suppose also that $M<\infty$. 
Then  there is a positive constant $\rho=\rho(n, r, p,\lambda, M, \theta, A, \eta)$ such that for all balls $ B\subset\Omega$ with radius smaller than $\rho$, and $u$ satisfying the conditions $u\in W_{0}^{2,r}(\Omega)$, $\|Lu\|_{L^{p), \lambda}_{\theta, A}}(B)<\infty$  we have that $u_{x_{i}x_{j}}\in L^{p), \lambda}_{\theta,A}(B)$ and, moreover, there exists a positive constant $C=C(n,p,\lambda, \theta, M, A,\eta)$ such that
\[
\norm{u_{x_{i}x_{j}}}_{L^{p), \lambda}_{\theta,A}(B)}\leqslant C\norm{ Lu}_{L^{p), \lambda}_{\theta, A}(B)}.
\]
\end{theorem}
\begin{proof}
It is easy to verify that $K_{ij}$ satisfies the condition in Corollary by the representation of $u_{x_{i}x_{j}}$ and the conditions of $K_{ij}$. Thus, from Corollary A, we deduce, for any $\varepsilon >0$,

\[
\norm{u_{x_{i}x_{j}}}_{ L^{ p), \lambda }_{\theta, A}(B)} \leqslant C \varepsilon \norm{u_{x_{i}x_{j}}}_{ L^{p), \lambda}_{\theta, A}(B)}+C\norm{Lu}_{L^{p), \lambda}_{\theta, A}(B)}.
\]
Choosing $\varepsilon$ to be small enough $(\mathrm{e}.\mathrm{g}.\ \varepsilon <1)$, we then obtain
\[
\norm{u_{x_{i}x_{j}}}_{L^{p), \lambda}_{\theta,A}(B)}\leq(C/(1- C\varepsilon))\norm{ Lu}_{L^{p), \lambda}_{\theta,A}(B)}.
\]
This finishes the proof.
\end{proof}

\section*{Acknowledgment}

 The first and second authors were  partially supported by the Shota Rustaveli National Science Foundation Grant (Project No. GNSF/ST09\_ 23\_ 3-100). The third author was partially supported by \textsl{Funda\c c\~ao para a Ci\^encia e a Tecnologia} (FCT), \textsf{Grant SFRH/BPD/63085/2009}, Portugal and by Pontificia Universidad Javeriana.








\end{document}